\documentclass{amsart}

\usepackage{amsthm}
\usepackage{mathtools}
\usepackage{epsfig}
\usepackage{graphicx}
\usepackage{caption}
\usepackage{subcaption}
\usepackage{color}

\newtheorem{theorem}{Theorem}

\newtheorem{proposition}[theorem]{Proposition}

\newtheorem{lemma}{Lemma}
\theoremstyle{definition}
\newtheorem{definition}{Definition}
\newtheorem{example}{Example}
\theoremstyle{remark}

\providecommand{\bL}{\mathbb {L}}
\providecommand{\F}{\mathbb {F}_q}
\providecommand{\G}{\Gamma}
\providecommand{\g}{\gamma}
\providecommand{\La}{\Lambda}

\providecommand{\D}{\mathcal{D}}

\providecommand{\K}{\mathbb F_q((t^{-1}))}

\begin{document}

\title[Farey map, Diophantine approximation and trees]{Farey map, Diophantine approximation and Bruhat-Tits tree}

\author{Dong Han Kim}
\address{Department of Mathematics Education, Dongguk University-Seoul, Seoul 100-715, Korea}
\email{kim2010@dongguk.edu}

\author{Seonhee Lim}
\address{Department of Mathematics, Seoul National University, Seoul 151-747, Korea}
\email{slim@snu.ac.kr}

\author{Hitoshi Nakada}
\address{Department of Mathematics, Keio University, Yokohama 223-8522, Japan}
\email{nakada@keio.ac.jp}

\author{Rie Natsui}
\address{Department of Mathematics, Japan Women's University, Tokyo 112-8681, Japan}
\email{natsui@fc.jwu.ac.jp}



\date{}

\begin{abstract}
Based on Broise-Alamichel and Paulin's work on the Gauss map corresponding to the principal convergents, we continue the study of the Gauss map via Farey maps to contain all the intermediate convergents.
We define the geometric Farey map, which is given by time-1 map of the geodesic flow. We also define algebraic Farey maps, better suited for arithmetic properties, which produce all the intermediate convergents.
Then we obtain the ergodic invariant measures for the Farey maps and the convergent speed. 
\end{abstract}

\maketitle

\section{Introduction}
Since Artin's work \cite{Art}, the relation between the continued fraction expansion and the geodesic flow on the modular surface $\mathbb{H}^2 / SL_2(\mathbb{Z})$ has been studied extensively (see \cite{KatUga} and references therein). 
The map that gives the continued fraction expansion of a given real number in $(0,1)$, which is defined on $(0,1) \subset \mathbb{R}$ to itself by $g: x \mapsto \{1/x \}$, the fractional part of $1/x$, is called the continued fraction map or the Gauss map. 
It is precisely the first return map of the geodesic flow on the surface 
$\mathbb{H}^2/\Gamma$, where $\Gamma$ is a subgroup of $SL_2(\mathbb{Z})$ of 
index 2, corresponding to the dual of Farey tessellation \cite{Series}. 

More recently, A. Broise-Alamichel and F. Paulin studied the relation between continued fraction expansion for functions and geodesic flow in trees, extending Artin's work to function fields \cite{BroisePaulin, Paulin}.

On the other hand, the Farey map for the real case was 
introduced to find intermediate convergents by S. Ito \cite{Ito} and also 
as an intermittent model by M. Feigenbaum, I. Procaccia, 
and T. Tel \cite{FeiProTel}, independently. It is defined by
$$ 
F_\textrm{real}(x) = \begin{cases} \dfrac{x}{1-x} , & \text{ if }  0 \le x < \frac 12, \\ 
\dfrac{1-x}{x} , & \text{ if }  \frac 12 \le x < 1. \end{cases}
$$
The Gauss map is an acceleration of the Farey map: $g(x) = {F_\textrm{real}}^{n(x)}(x),
$
where $n(x)$ is the first partial quotient in the continued fraction expansion of $x$ (see \cite{Series} for the geometric meaning of the
Farey map).

In this article, we define two kinds of new maps, one which we call geometric Farey map, and the others which we call algebraic Farey maps for function fields. We investigate arithmetic properties of the intermediate convergents 
arising from these maps, and study ergodic properties of these maps.

Although a Farey map for function fields was constructed 
already in \cite{BerNakNat} based on the Euclidean algorithm for polynomials, 
we propose alternative definitions. 

We first construct the geometric Farey map, which is a geometric analogue of the Farey map for the real case. Namely, it is the time-one map of the geodesic flow on the modular ray (see Section~\ref{gfarey} for the detail). Unfortunately, the analogy is not so clear, since  there are many Ford spheres (see Section~\ref{sec:2.3} for definition) tangent to each other in a tree, unlike the real case. This leads us to define algebraic Farey maps.

The algebraic Farey maps has many advantages. 
First of all, the Farey map in \cite{BerNakNat} is either a special case or some acceleration of an algebraic Farey map (see Section~\ref{sec:3.2}).
It is roughly a composition of two time-one maps of the geodesic flow.
The map is given by the multiplication action of an element of $SL_2(\F[t])$ which preserves Ford spheres.

As in the real case, the Artin map is an acceleration of an algebraic Farey map by the time depending only on the degree of the first partial quotient. Moreover, we obtain all the intermediate convergents (see Section~\ref{sec:3.2}).

Note that the first return map of the geodesic flow was used in \cite{BroisePaulin} in relation with the Gauss map, 
but that no Farey map has been investigated 
geometrically 
yet. 

Let $\F(t)$ be the quotient field of polynomials over the finite field $\F$ of $q$ elements, where $q$ is a power of a prime.
Denote by $K$ the completion of $\mathbb F_q(t)$ with respect to the valuation $\nu_\infty(P/Q) = -\deg P + \deg Q$ and $\mathcal O$ be the corresponding discrete valuation ring. Then $K$ is the field of formal Laurent series
\[
K = \mathbb F_q ((t^{-1})) = \{ f = a_n t^n + \dots + a_1 t + a_0 + a_{-1} t^{-1} + \cdots : a_i \in \mathbb F_q \}.
\]
We denote by $|f|$ the absolute value, i.e.,  $|f| = q^{\deg (f)}=q^n$ for $f = \sum_{ -\infty < i \le n} a_i t^i \in \mathbb F_q ((t^{-1}))$, $a_n \ne 0$, with convention $\deg(0) =  -\infty$. 
Note that $\mathcal O = \{ f \in K : |f| \le 1\}$ and $K$ is non-Archimedean as $|f + g| = \max( |f|,|g|)$. 

Just as the Gauss map is defined for real numbers between zero 
and one, we restrict ourselves to the subset $\mathbb L=\{ f \in K : | f | < 1 \} = t^{-1} \mathcal{O}$ of $K$.
For $f \in {\mathbb L}$ and a polynomial $Q$, there exists a unique
polynomial $P$ such that $\deg (Qf - P) < 0$.  We put $\{ Qf \} = Qf - P$ for such $P$.
Now the Artin map, the analogue of the Gauss map, 
is defined as 
\[ \Psi : f \mapsto \{1/f \} \;\; \mathrm{for} \;\; f \in \mathbb L -\{0 \}. \]
Artin map was studied extensively in \cite{Paulin}, \cite{BroisePaulin}. They showed that the Artin map is the first return map of the geodesic flow on the modular surface which is a quotient of Bruhat-Tits tree by a lattice subgroup.

We first define the \textit{geometric Farey map $F$} on $\mathbb L \times \mathbb Z$ as  
$$F(f, n) = \begin{cases}
\left(tf - [tf], n+1\right),   &\text{ if } \deg(f) < -1 \text{ or } n < 0, \\
\left( \dfrac1{tf} - \left[\dfrac{1}{tf}\right], -(n+1) \right),    &\text{ if } \deg(f) = -1 \text{ and } n \ge 0
\end{cases}$$
and show that it is the time-one map of the geodesic flow.

Now let  $\mu$ be the Haar measure of $\mathbb F_{q}((t^{-1}))$ 
normalized as $\mu (\mathcal O) = 1$.
\begin{theorem}\label{thm1}
Let $\mu_G$ be the measure on $\mathbb L \times \mathbb Z$ defined as follows: for each measurable $\mathbf E \subset \mathbb L$,
$$\mu_G(\mathbf E \times \{n\}) = \begin{cases} \dfrac{q-1}{2q^{n}} \mu (\mathbf E), &\text{ for } n \ge 0, \\
\dfrac{q-1}{2q^{-n-1}} \mu (\mathbf E), &\text{ for } n < 0. \end{cases} $$
Then $\mu_G$ is an ergodic invariant measure for the geometric Farey map $F$.
\end{theorem}

Next, we combine two time-one maps of the geodesic flow to define an algebraic Farey map and generalize to a family of maps depending on a function:
for a given $h \in {\mathbb L}$ with $\deg(h) = -1$,
the \textit{algebraic Farey map $F_h$ associated to $h$} is defined by 
$$ F_h (f) = \begin{cases}
\dfrac{f}{1-[(1-h) f^{-1}]f},  \quad & \deg (f) \le -1, \\
\dfrac{1-[ (1-h) f^{-1}]f}{f},  & \deg (f) = 0. \\
\end{cases}$$
In Section~\ref{sec:3.2}, we use the above Farey map to construct intermediate convergents and show that we obtain a nice Diophantine property. We also show that if a rational function satisfy a better Diophantine property, then it is an intermediate convergent constructed by algebraic Farey maps.

For each $n \ge 0$, denote  
$\mathbb J_n = \{  f \in \mathcal O : \deg(f) = - n \}.$ 
Define a measure $\mu_A$ on $\mathcal O$ by 
$$
\mu_A ({\bf D}) = \frac{q^2}{2q-1} \cdot \mu ({\bf D} \cap \mathbb L) + \frac{q}{2q-1} \cdot \mu ( {\bf D} \cap \mathbb J_0)
$$
for each Borel set ${\bf D} \subset \mathcal O$. 

\begin{theorem}\label{thm2}
For each $h$, the probability measure $\mu_A$ on $\mathcal O$ is an ergodic invariant measure for the algebraic Farey map $F_h$.
\end{theorem}

Let $U_\ell/V_\ell$ be an intermediate convergent. Our last theorem is the following theorem of convergence rate.
\begin{theorem}\label{thm3}
For $\mu$-a.e. $f$, we have 
\[
\lim_{\ell \to \infty}
\frac{1}{\ell} \log_{q} \left| f \, - \, \frac{U_{\ell}}{V_{\ell}} \right| 
\, = \, -\frac{2q}{2q - 1}.
\]
\end{theorem}

\section{Preliminary : continued fraction expansion and geodesic flow on Bruhat-Tits tree of $SL_2(K)$}

Before defining Farey maps in the next section, 
let us recall 
Paulin's geometric interpretation \cite{Paulin} of the Artin map in this 
section.

Let $G = SL_2(K)= SL_2(\mathbb F_q ((t^{-1})))$ and  $\Gamma =SL_2(\mathbb F_q [t])$ be the modular group of Weil.


\subsection{Bruhat-Tits tree}\label{sec:2.1}

\textit{The Bruhat-Tits tree} $T_q$ of $G$ is a $(q+1)$-regular tree whose vertex set is the set of homothety classes (by $K^\times $) of $\mathcal{O}$-lattices in $V= K \times K$, i.e., classes of rank-$2$ free $\mathcal{O}$--submodules that generate $V$ as a vector space.
Two vertices $\La$ and $\La'$ have a common edge if and only if there exist representatives $L$, $L'$ of $\La$ and $\La'$ such that $L' \subset L$ and $L/L'$ is isomorphic to $\mathcal{O}/t^{-1} \mathcal{O} = \F$. 
Denote by $ \{e_1 = ( \begin{smallmatrix} 1 \\ 0 \end{smallmatrix} ) , e_2 = \left( \begin{smallmatrix} 0 \\ 1 \end{smallmatrix} \right) \}$ the canonical basis of $V$.
For $a,b,c,d \in K$, we denote by $[\begin{smallmatrix} a & b \\ c & d \end{smallmatrix}]$ the class of the lattice 
$$ \begin{pmatrix} a & b \\ c & d \end{pmatrix} 
\left( e_1 \mathcal O \oplus e_2 \mathcal O \right) =  
\begin{pmatrix} a \\ c \end{pmatrix} \mathcal O \oplus  \begin{pmatrix} b \\ d \end{pmatrix} \mathcal O .$$
Note that there are many ways to express a vertex by such a matrix as the stabilizer of a vertex is isomorphic to $SL_2(\mathcal{O})$.  
Let us denote the vertex of the standard lattice class $\left[ \begin{smallmatrix} 1 & 0 \\ 0 & 1 \end{smallmatrix} \right]$ by $x_*$.

The metric on $T_q$ is given by assigning length $1$ to every edge. 
A \textit{geodesic ray} is an isometry $[0, \infty[ \to T$. 
The \textit{(Gromov) boundary} $\partial  T_{q}$ of 
$T_{q} $ is defined as the set of equivalence classes of geodesic rays where two geodesic rays are equivalent if and only if their intersection is still a geodesic ray.

\subsection{Action of $G$ and $\G$ on the boundary of Bruhat-Tits tree} 


The action of $G=SL(V)$ on $T_q$ defined by $g [L] = [g L] $, for $g \in G$, is well-defined. This action is transitive on the set of edges and extends naturally to $T_q \cup \partial T_q$ as well. 

By this action, the boundary $\partial T_q$ can be identified with the projective line $\mathbb P^1 (K) = K \cup \{ \infty \}$. 
More precisely, for a given equivalence class of the geodesic rays, we choose a representative geodesic ray with vertices $[L_n]$ such that $L_{n+1} \subset L_n$. The associated element of $\mathbb P ^1 (K)$ is the class of the unique line that contains the intersection of $L_n$'s.

For $n \in \mathbb{Z}$, let $L_n$ be the $\mathcal{O}$-lattice 
with basis $\{ t^n e_1, e_2 \}$ and $\Lambda_n=\left[\begin{smallmatrix} t^n & 0 \\ 0 & 1 \end{smallmatrix}\right]$ be the corresponding vertex in the tree $T_q$. 
The geodesic ray $\D_\infty$ with vertices $\La_n$, $n \ge 0$, (strictly speaking a quotient graph of groups with graph $\D_\infty$), called \textit{the fundamental ray of $\G$}, is a fundamental domain for the action of $\G$ on $T_q$, i.e., the orbit of $\D_\infty$ under $\G$-action cover $T_q$. 

\begin{lemma}\label{prop:action}
The action of $G$ on $\partial T_q$ corresponds to the action of $G$ by homographies on $K \cup \{ \infty \}$. 
Its restriction to $K$ is as follows : if $f \in K$ and $\g = (\begin{smallmatrix} a & b \\ c & d \end{smallmatrix}) \in SL(V)$, then $\g \circ f = (af+b) / (cf+d)$.
\end{lemma}

\begin{proof}
Denote by $\D_0$ the geodesic ray from $x_*$ to $0 \in \partial T_q$, which have vertices $\La_{-n} =\left[\begin{smallmatrix} 1 & 0 \\ 0 & t^{n} \end{smallmatrix}\right]$, for $n \ge 0$.
For $g \in G$, we denote by $g \D_0$ the geodesic ray with vertices $\{ g \left[\begin{smallmatrix} 1 & 0 \\ 0 & t^{n} \end{smallmatrix}\right] \}_{n \ge 0}$.
Consider the geodesic ray $(\begin{smallmatrix} 1 & f \\ 0 & 1 \end{smallmatrix})\D_0$ to $f$.
Then we have
$\g \circ (\begin{smallmatrix} 1 & f \\ 0 & 1 \end{smallmatrix}) \D_0 = (\begin{smallmatrix} a & af+b \\ c  & cf+d \end{smallmatrix}) \D_0 ,$
which is a geodesic ray to $\g \circ f = (af+b) / (cf+d)$.
\end{proof}


The action of $G$ is transitive on the set of triplets of points on the boundary $\partial T_q$ ($G/\{\pm Id\}$ acts simply transitively). The orbit of $\infty$ under $\Gamma$ is $\mathbb{F}_q(t) \cup  \{\infty \}$.


\subsection{Horospheres and Ford spheres}\label{sec:2.3} 

A \textit{Buseman function at $\omega \in \partial T_q$} is the map $\beta_\omega : T_q \times T_q \to \mathbb R$ defined by 
$ \beta_\omega (x,y) = \lim_{t \to \infty}  [ d(y, c(t)) - d(x, c(t)) ],$
where $c(t)$ is any geodesic ray converging to $\omega$. It is independent of the choice of $c(t)$.

A \textit{horosphere based on $\omega$} is a level set of Buseman function $\beta_\omega$. By the cocycle relation $\beta_\omega(x,y) + \beta_\omega(y,z) = \beta_\omega(x, z)$, two points are in the same horosphere based on $\omega$ if and only if $\beta_\omega(x,y)=0$.
A \textit{horoball based on $\omega$} is the interior of a horosphere based on $\omega$, i.e. the union of all geodesic rays from a point on the horosphere $H$ to $\omega$.

Each vertex in the horosphere $H_{\infty,n}$ based on $\infty \in \partial T_q$ passing by $\La_n$  can be uniquely represented as
$$ \begin{bmatrix} 1 & A \\ 0 & t^{-n} \end{bmatrix}, \qquad A \in t \F[t]. $$
Consult \cite{Paulin} for details.
Note that the geodesic $]\infty,f[$ connecting $\infty$ and $f \in \partial T_q$ has vertices
$\left(\begin{smallmatrix} 1 & f \\ 0 & 1 \end{smallmatrix}\right) \La_{n} = \left[\begin{smallmatrix} 1 & f t^n \\ 0 & t^{n} \end{smallmatrix}\right]$, which approaches $f$ as $n$ goes to $-\infty$.

Consider the horosphere $H_\infty = H_{\infty,0}$ based on $\infty \in \partial T_q$ passing by $x_*$. 
We have $H_\infty = \G_\infty x_*$, where $\G_\infty = \textrm{Stab}_{\G}(\infty)$.
Let us denote by $HB_\infty$ the interior of $H_\infty$, which is the orbit of the fundamental ray $\D_\infty$ by $\G_\infty$.

\begin{definition} 
A \textit{Ford sphere} is a horosphere of the form $H_\g = H_{\g\infty} = \g H_\infty$ for some $\gamma \in \Gamma$, i.e. a horosphere based on a point in $\F(t)  \cup \{ \infty \}.$ In other words, $H_\frac PQ = \left \{ \g x_* : \g \in \G,  \g(\infty) = \frac PQ \right \}$.
A \textit{Ford ball} is a horoball of the form $HB_\g = HB_{\g \infty}$. The Ford ball $HB_\g$ is said to be \textit{associated to} $H_\g$.
\end{definition}

Ford spheres form a maximal $\G$-equivariant family of horoballs with disjoint interior (see Section 6.2 of \cite{Paulin}).
For Ford circles and Ford spheres in number fields, see \cite{Ford} (also \cite{Nakada}).


Let $H$ be a horosphere based on $\omega$. 
For all points $u\neq v$ in $\partial T_q - \{ \omega \}$, the geodesic $]u,\omega[$ intersects $H$ in one point $h$ and it intersects the geodesic $]\omega, v[$ in one geodesic ray $]\omega, p ]$. 
We denote by $(u,v)_{\omega, H}$ the algebraic distance from $h$ to $p$ if $u \neq v$, and $\infty$ otherwise. Now we call
$$ d_{\omega, H} (u,v) = q^{-(u,v)_{\omega, H}}$$
the \textit{Hamenst\"adt distance} on $\partial T_q - \{ \omega \}$.
 It is an ultrametric and $d_{\omega, H'} (u,v) = q^{-\beta_\omega(H', H)} d_{\omega, H} (u,v)$, where $\beta_\omega(H',H) = \beta_\omega(x',x)$ for any $x' \in H' , x \in H$. The following lemma follows immediately.

\begin{lemma}\label{Hamdist}
Let $\omega, \omega' \in \partial T_q$ and $H, H'$ be horospheres based on $\omega$, $\omega'$.
Denote by $p, p'$ the points of intersection of the geodesic $]\omega, \omega'[$ and $H, H'$ respectively.
Then we have
$$ d_{\omega, H} (\omega',v) \cdot d_{\omega', H'} (\omega,v) \cdot q^{\beta_\omega (p,p')} = 1,$$
for any $v \in \partial T_q$.
\end{lemma}



\subsection{Diophantine approximation and Artin map}\label{sec:2.2.3}


Let us consider geodesics starting from $\infty$.
Each geodesic that comes in $HB_\g$ for some $\g$ either goes out at some finite time or converges to a rational point, namely the base point of $HB_\g$. Thus any geodesic whose end point is not rational has a finite first return time to $\G x_* = \{ \g x_* : \g \in \G \}$, since the exit vertex, being on $H_\g$, belongs to $\gamma \G_\infty x_* \subset \G x_*$. 




\begin{figure}[h]
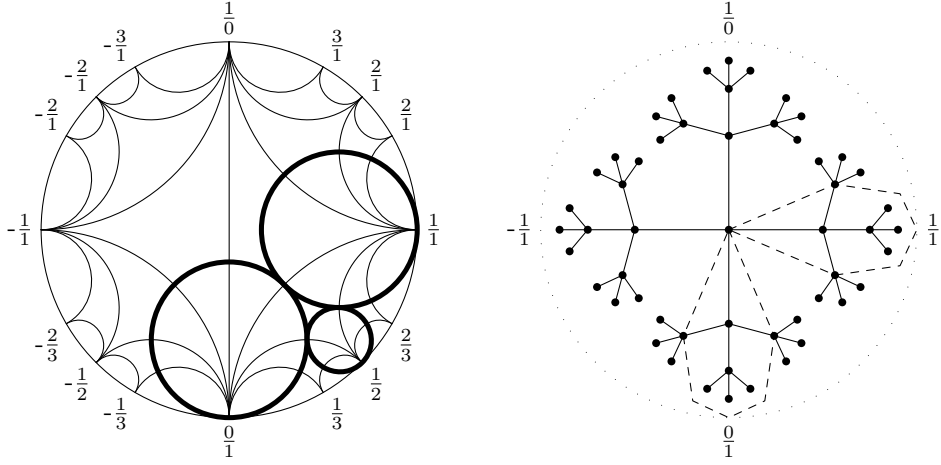

\includegraphics{diagram-1.mps}
\qquad 
\includegraphics{tree-1.mps}
\caption{Ford circles in a plane and Ford spheres in a tree}
\end{figure}


For a given irrational $f \in K$,
let $A_k$ be the partial quotients of $f$ and let
$$\frac{P_k}{Q_k} = A_0 + \dfrac{1}{A_1 + \dfrac{1}{A_2 + \dfrac{1}{\ddots + 1/A_k}}}, \qquad (P_k, Q_k) = 1, $$
be the $k$-th principal convergent of $f$.  
The recurrence relation reads: $Q_{k+1} = A_{k+1} Q_k + Q_{k-1}$ and $\deg(Q_{k+1}) - \deg(Q_k) = \deg(A_{k+1})$ for $k \ge 1$ with $Q_0 = 1$. 

Note that for any $k \ge 1$,
\begin{equation*} | \{ Q_{k} f \} | = \frac{1}{|Q_{k+1}|}.\end{equation*}
See \cite{BeNa}, \cite{Sch} and the references therein for details.

\begin{lemma}\cite{Paulin}\label{lem:2} The sequence of Ford balls intersected by the geodesic from $\infty$ to $f \in K$ is the sequence of Ford balls based on the principal convergents of $f$.
\end{lemma}





For a given rational $\frac PQ \in \mathbb F_q(t)$, by Lemma~\ref{Hamdist} we have
$$ \left |\frac PQ - f \right| \cdot d_{\frac PQ, H} (\infty, f) \cdot q^{\beta_\infty (p,p')} = 1,$$
where $p, p'$ be the intersecting point of the geodesic $]\infty, \frac PQ[$ and horospheres $H_\infty$, $H=H_{\frac PQ}$ respectively.
We have $ q^{\beta_\infty (p,p')}  = |Q|^2$.

According to whether $]\infty, f[$ intersects $H_\frac PQ$ (i.e. two vertices in $H_\frac PQ$ belong to $]\infty, f[$), tangent to $H_\frac PQ$ (i.e., exactly one vertex in $H_\frac PQ$ belongs to $]\infty, f[$), or disjoint from $H_\frac PQ$, we have
$$ d_{\frac PQ, H} (\infty, f) > 1, \quad d_{\frac PQ, H} (\infty, f) = 1, \quad d_{\frac PQ, H} (\infty, f) < 1, $$
respectively, i.e., 
$$ \left| f - \frac PQ \right| < \frac{1}{|Q|^2}, \ \left| f - \frac PQ \right| = \frac{1}{|Q|^2}, \ \left| f - \frac PQ \right| > \frac{1}{|Q|^2},$$
respectively.

Since the Ford sphere tiles the tree $T_q$, the geodesic $]\infty, f[$ intersects infinity many Ford spheres if $f$ is not of the form $\frac PQ$. 
Hence, we obtain Hurwitz's theorem for the formal power series: 
 for any $f \in \K - \F(t)$. there are infinitely many $\frac PQ$'s such that 
$$ \left| f - \frac PQ \right| < \frac{1}{|Q|^2}.$$


\section{Geometric Farey map}

Based on Section~\ref{sec:2.2.3}, we define the geometric Farey map as the time-one map of the geodesic flow.
Let us denote the polynomial part of $f$ by $[f]$ :
$$[a_n t^n + a_{n-1}t^{n-1} + \cdots + a_0 + a_{-1} t^{-1} + \cdots ] = a_n t^n + \cdots + a_0. $$

\subsection{Geometric Farey map and convergents}\label{gfarey}

For each $f \in \bL$, consider the geodesic ray $(\begin{smallmatrix} 1 & f \\ 0 & 1 \end{smallmatrix}) \D_0$ from $x_*$ to $f$ with vertices consisting of $\{ (\begin{smallmatrix} 1 & f \\ 0 & 1 \end{smallmatrix})  \La_{-n}\}_{n \ge 0} =  \{ ( \begin{smallmatrix} 1 & f \\ 0 & 1 \end{smallmatrix})  \left[\begin{smallmatrix} 1 & 0 \\ 0 & t^{n} \end{smallmatrix}\right] \}_{n \ge 0}$.
Then we have
\begin{align*}
\begin{pmatrix} 1 & -t^{-1}[tf] \\ 0 & t^{-1} \end{pmatrix} \begin{pmatrix} 1 & f \\ 0 & 1\end{pmatrix} \Lambda_{-1} &=  \begin{bmatrix} 1 & t f -[tf] \\ 0 & 1 \end{bmatrix} = \begin{bmatrix} 1 & 0 \\ 0 & 1 \end{bmatrix} = \begin{pmatrix} 1 & f \\ 0 & 1\end{pmatrix} \Lambda_{0},\\
\begin{pmatrix} 1 & -t^{-1}[tf] \\ 0 & t^{-1} \end{pmatrix} \begin{pmatrix} 1 & f \\ 0 & 1\end{pmatrix} \Lambda_{-2} &=  \begin{bmatrix} 1 & t^2 f -t[tf] \\ 0 & t \end{bmatrix} = \begin{bmatrix} 1 & 0 \\ 0 & t \end{bmatrix} = \begin{pmatrix} 1 & f \\ 0 & 1\end{pmatrix} \Lambda_{-1},
\end{align*}
and if $\deg(f) = -1$, then we have
\begin{align*}
\begin{pmatrix} -\left[\frac{1}{tf}\right] & t^{-1} \\ 1 & 0 \end{pmatrix} \begin{pmatrix} 1 & f \\ 0 & 1\end{pmatrix} \Lambda_{-1}
&= \begin{bmatrix} - \frac{1}{[tf]} & 1 \\ 1 & 0  \end{bmatrix}
= \begin{bmatrix} 1 & 0 \\ 0 & 1 \end{bmatrix} = \begin{pmatrix} 1 & f \\ 0 & 1\end{pmatrix} \Lambda_{0}, \\
\begin{pmatrix} -\left[\frac{1}{tf}\right] & t^{-1} \\ 1 & 0 \end{pmatrix} \begin{pmatrix} 1 & f \\ 0 & 1\end{pmatrix} \Lambda_{-2}
&= \begin{bmatrix} - \frac{1}{[tf]} & 0 \\ 1 & t[tf]  \end{bmatrix} 
= \begin{bmatrix} 1 & 0 \\ 0 & t \end{bmatrix} = \begin{pmatrix} 1 & f \\ 0 & 1\end{pmatrix} \Lambda_{-1}.
\end{align*}
Thus, the left multiplication by $ \left( \begin{smallmatrix} 1 & -t^{-1}[tf] \\ 0 & t^{-1} \end{smallmatrix} \right)$ or $\left( \begin{smallmatrix} -\left[\frac{1}{tf}\right] & t^{-1} \\ 1 & 0 \end{smallmatrix} \right)$ for an $f$ with $\deg f = -1$
on geodesic ray $(\begin{smallmatrix} 1 & f \\ 0 & 1 \end{smallmatrix}) \D_0$  can be considered as a time-one map of the geodesic ray.
By these maps the geodesic ray to $f$ is sent to the geodesic ray to
$$\frac{f - t^{-1}[tf]}{t^{-1}} = tf - [tf] \ \text{ or } \ \frac{-f \left[\frac{1}{tf}\right] + t^{-1}}{f} = \frac{1}{tf} - \left[\frac{1}{tf}\right].$$
Therefore, we define the geometric Farey may as follows.
\begin{definition}[Geometric Farey map]
We define the geometric Farey map $F$ on $\mathbb L \times \mathbb Z$ onto itself by 
$$F(f, n) = \begin{cases}
\left(tf - [tf], n+1\right),   &\text{ if } \deg(f) < -1 \text{ or } n < 0, \\
\left( \dfrac1{tf} - \left[\dfrac{1}{tf}\right], -(n+1) \right),    &\text{ if } \deg(f) = -1 \text{ and } n \ge 0.
\end{cases}$$
\end{definition}

For each $f \in \mathbb L$, we have
$$ F^{-2\deg(f)} (f,0) = \left( \psi (f), 0 \right),$$

\begin{example}
Let $$f = \dfrac{1}{2t^3 + t^2 + 2 + r}, \quad \deg(r) < 0.$$
Then we have
\begin{align*}
F^1(f,0) &= \left( \frac{t}{2t^3 + t^2 + 2 + r}, 1 \right), &
F^2(f,0) &= \left( \frac{t^2}{2t^3 + t^2 + 2 + r}, 2 \right), \\
F^3(f,0) &= \left( \frac{t^2 + 2 + r}{t^3}, -3 \right), &
F^4(f,0) &= \left( \frac{2 + r}{t^2}, -2 \right),  \\
F^5(f,0) &= \left( \frac{2 + r}{t}, -1 \right), &
F^6(f,0) &= \left( r, 0 \right).
\end{align*}
\end{example}


Let $M(f, n)$ be the matrix defined by  
$$ M(f, n) = \begin{cases}
\begin{pmatrix} 1 & -t^{-1}[tf] \\ 0 & t^{-1} \end{pmatrix}^{-1} = \begin{pmatrix} 1 & 0 \\ 0 & t \end{pmatrix}, &\text{if } \deg(f) < -1 \text{ and } n \ge 0, \\
\begin{pmatrix} -\left[\frac{1}{tf}\right] & t^{-1} \\ 1 & 0 \end{pmatrix}^{-1} = \begin{pmatrix} 0 & 1 \\ t & t\left[\frac{1}{tf}\right] \end{pmatrix}, &\text{if }  \deg(f) = -1 \text{ and } n \ge 0,\\
\dfrac 1t \begin{pmatrix} 1 & -t^{-1}[tf] \\ 0 & t^{-1} \end{pmatrix}^{-1} = \begin{pmatrix} t^{-1} & t^{-1} [tf] \\ 0 & 1 \end{pmatrix}, &\text{if } n < 0.
\end{cases}$$

For each $f \in \mathbb L$, 
if $\ell = 2\deg(A_1) + \dots + 2\deg(A_k) + i$, $0 \le i < \deg(A_{k+1})$, then 
$$ M(f, 0) \cdots M(F^{\ell-1}(f, 0)) 
= \begin{pmatrix} P_{k-1} & P_k \\ Q_{k-1} & Q_k \end{pmatrix} 
\begin{pmatrix} 1 & 0 \\ 0 & t^{i} \end{pmatrix}= \begin{pmatrix} P_{k-1} & t^{i} P_k \\ Q_{k-1} & t^{i} Q_k \end{pmatrix}. $$
If $\ell = 2\deg(A_1) + \dots + 2\deg(A_k) + \deg(A_{k+1})+ i$, $0 \le i \le \deg(A_{k+1})$, then 
\begin{equation*} \begin{split} M(f, 0) \cdots M(F^{\ell-1}(f, 0)) 
&= \begin{pmatrix} P_{k-1} & P_k \\ Q_{k-1} & Q_k \end{pmatrix} 
\begin{pmatrix} 0 & 1 \\ t^{m-i} & a_m t^{m} + \dots + a_{m-i}t^{m-i} \end{pmatrix} \\
&= \begin{pmatrix} t^{m-i}P_{k} & (a_m t^{m} + \dots + a_{m-i}t^{m-i} )P_k + P_{k-1} \\ t^{m-i}Q_{k} & (a_m t^{m} + \dots + a_{m-i} t^{m-i} )Q_k + Q_{k-1} \end{pmatrix},
\end{split} \end{equation*}
where $A_{k+1} = a_m t^m + \dots + a_1t + a_0$.

By applying the Farey map $\ell$-times, the first vertex of geodesic $[\begin{smallmatrix} 1 & 0 \\ 0 & t \end{smallmatrix} ]$ sent to the vertex represented by the matrix $ M(f, 0) \cdots M(F^{\ell-1}(f, 0)) $.
The geodesic $\begin{pmatrix} t^{m-i}P_{k} & (a_m t^{m} + \dots + a_{m-i}t^{m-i} )P_k + P_{k-1} \\ t^{m-i}Q_{k} & (a_m t^{m} + \dots + a_{m-i} t^{m-i} )Q_k + Q_{k-1} \end{pmatrix} \D_0$
has the limit point 
$$ \frac{(a_m t^m + \dots + a_{m-i}t^{m-i} )P_k + P_{k-1}}{(a_m t^m + \dots + a_{m-i}t^{m-i} )Q_k + Q_{k-1}}.$$
Therefore, the Hamenst\"adt distance between $f$ and $\frac{(a_m t^m + \dots + a_{m-i}t^{m-i} )P_k + P_{k-1}}{(a_m t^m + \dots + a_{m-i}t^{m-i} )Q_k + Q_{k-1}}$ is less than or equal to $q^{-\ell}$.
Hence, we have
$$
\left| f - \frac{(a_m t^m + \dots + a_{m-i}t^{m-i} )P_k + P_{k-1}}{(a_m t^m + \dots + a_{m-i}t^{m-i} )Q_k + Q_{k-1}} \right|  \le q^\ell = \frac{q^{-i}}{|Q_k|\cdot| Q_{k+1}|} .
$$

\subsection{Ergodic theory of the geometric Farey map}

Let $\mu$ be the Haar measure of $\mathbb F_{q}((t^{-1}))$ 
normalized as $\mu (\mathcal O) = 1$.
Then, as it is stated in the introduction,
The measure $\mu_G$ on $\mathbb L \times \mathbb Z$ defined by for each measurable $\mathbf E \subset \mathbb L$
$$\mu_G (\mathbf E \times \{n\}) = \begin{cases} \dfrac{q-1}{2q^{n}} \mu (\mathbf E), \text{ for } n \ge 0, \\
\dfrac{q-1}{2q^{-n-1}} \mu (\mathbf E), \text{ for } n < 0 \end{cases} $$
is an invariant measure for the geometric Farey map $F$.

\begin{proof}[Proof of Theorem~\ref{thm1}]
For each measurable $\mathbf E \subset \mathbb L$, if $n > 0$, we have
\begin{equation*}\begin{split}
\mu_G( F^{-1} ( \mathbf E \times \{n\} )) 
&= \mu_G( t^{-1} \mathbf E \times \{n-1\} )
= \frac{q-1}{2q^{n-1}} \mu( t^{-1} \mathbf E ) \\
&= \frac{q-1}{2q^{n}} \mu( \mathbf E )
= \mu_G( \mathbf E \times \{n\})
\end{split} \end{equation*}
and if $n = 0$, we have
\begin{equation*}\begin{split}
\mu_G( F^{-1} ( \mathbf E \times \{ 0 \} )) 
&= \mu_G \left( \bigcup_{a \in \F} \left( a + t^{-1} \mathbf E \right) \times \{-1\} \right) 
= q \cdot \mu_G \left( t^{-1} \mathbf E \times \{-1\} \right) \\
&= \mu_G \left( t^{-1} \mathbf E \times \{-1\} \right) 
= \frac{q-1}{2} \cdot \mu( \mathbf E ) 
= \mu_G ( \mathbf E \times \{ 0 \}).
\end{split} \end{equation*}

Suppose $n < 0$. Then since for each measurable $\mathbf E \subset \mathbb L$ and $a \in \mathbb F_{q}^*$
\[
\mu ( (a + {\bf E})^{-1} ) = \mu (\mathbf E),
\]
we have
\begin{equation*}\begin{split}
\mu_G( F^{-1} ( \mathbf E \times \{ n \} )) 
&= \mu_G \left( \bigcup_{a \in \F} \left( a + t^{-1} \mathbf E \right) \times \{n-1\} \right) \\
& \qquad + \mu_G \left( \bigcup_{a \in \F^*} t^{-1} (a + {\mathbf E})^{-1} \times \{-n-1\} \right) \\
&= q \cdot \mu_G \left( t^{-1} \mathbf E \times \{n-1\} \right) + \frac{q-1}{q} \cdot \mu_G \left( (1 + {\mathbf E})^{-1} \times \{-n-1\} \right) \\
&= \mu_G \left( \mathbf E \times \{n-1\} \right) + \frac{q-1}{q} \cdot \frac{q-1}{2q^{-n-1}} \cdot \mu \left( (1 + {\mathbf E})^{-1} \right) \\
&= \frac{q-1}{2q^{-n}} \cdot \mu ( \mathbf E ) + \frac{q-1}{q} \cdot \frac{q-1}{2q^{-n-1}} \cdot \mu ({\mathbf E}) 
= \mu_G ( \mathbf E \times \{ n \}).
\end{split} \end{equation*}

The ergodicity follows the fact that the Artin map is a jump transformation of $F$ and that the Artin map is ergodic with respect to the Haar measure $\mu$, see \cite{Schw}.
\end{proof}

\section{algebraic Farey map}\label{sec:4}
In this section, we define another family of Farey maps, which we call algebraic Farey maps, more suitable to obtain intermediate convergents. 
In the special case of $h=t^{-1}$, the Farey map $F_h$ is a slight modification of the geometric Farey map.

As was mentioned in the introduction, we define
Farey maps, for which the Farey map of Berth\'e, Nakada and Natsui \cite{BerNakNat} is either a special case or an accelleration of our Farey map. Let 
us first define intermediate convergents and algebraic Farey maps, and explain the geometrical and 
dynamical motivation.

\subsection{Intermediate convergents}\label{sec:3.2}
Recall that intermediate convergents in the real case are defined as rational numbers of the form
$ (a p_k + p_{k+1})/(aq_k + q_{k-1}), \;\; \; 0 < a < a_{k+1}.$
Alternatively, by letting $b= a_{k+1} -a$ and using the recursive relations $p_{k+1} = a_{k+1}p_k + p_{k-1}$,
it is equivalent to 
$$ \frac{p_{k+1} - bp_k}{q_{k+1} - bq_k}, \;\;\; 0 < b < a_{k+1}.$$
In analogy with the real case, we define intermediate convergents in the function field case as follows.
\begin{definition} The intermediate convergents are rational functions of the form
$$ \frac{P_{k+1} - BP_k}{Q_{k+1} - BQ_k}, \ B \in \F[t] \text{ with }  0 < |B| < | A_{k+1}|.$$
\end{definition}

\begin{theorem}\label{thm:3.1}
For $B \in \F[t]$ with $|B| \le |A_{k+1}|$, we have 
$$
\left| f - \frac{P_{k+1} - BP_k}{Q_{k+1} - BQ_k} \right| = \frac{{|B|}}{|Q_{k+1}|^2}
$$

If $U/V \in \F(t)$ with $\deg(Q) = \deg(Q_{k+1})$ satisfies
$$
\left| f - \frac{U}{V} \right| < \frac{1}{|Q_{k+1}|\cdot |Q_k|},
$$
then we have 
$$\frac UV = \frac{P_{k+1} - BP_k}{Q_{k+1} - BQ_k}, $$ for some $B \in \F[t]$ with $|B| < |A_{k+1}|$.
\end{theorem}

\begin{proof}
We have 
\begin{align*}
\left| f - \frac{P_{k+1} - BP_k}{Q_{k+1} - BQ_k} \right|  &= \frac{| (Q_{k+1} - BQ_k)f - (P_{k+1} - BP_k)|}{| Q_{k+1} - BQ_k|} \\
&=\frac{ | (Q_{k+1} f - P_{k+1}) +B(P_k - Q_kf)|}{ |Q_{k+1}|} \\
&= \frac{|B||Q_k f-P_k|}{|Q_{k+1}|} = \frac{|B|}{|Q_{k+1}|^2}.
\end{align*}


By the division algorithm, we have $V = a Q_{k+1} + B_{k+1} Q_{k} + \dots + B_{s+1} Q_s $ for some $s \ge 0$,
where $a \in \F^*$, $|B_{i+1}| < |A_{i+1}|$, $s \le i \le k$ and $B_{s+1} \ne 0$.
It follows that \begin{align*}
|\{ Vf \}| &= |\{ a Q_{k+1}f + B_{k+1} Q_{k}f + \dots + B_{s+1} Q_s f \}| \\
&= |\{B_{s+1} Q_s f \}| = \frac {|B_{s+1}|}{|Q_{s+1} |} .
\end{align*}
From the assumption, we have $\left|\{ V f \} \right| <  |Q_k|^{-1}$, which implies that $s = k$.
\end{proof}

\subsection{Algebraic Farey maps on the function field} 

\begin{definition}
For a given $h \in {\mathbb L}$ with $\deg(h) = -1$,
the \textit{Farey map $F_h$ associated to $h$} is defined as 
$$ F_h (f) = \begin{cases}
\dfrac{f}{1-[(1-h) f^{-1}]f},  \quad & \deg (f) \le -1, \\
\dfrac{1-[ (1-h) f^{-1}]f}{f},  & \deg (f) = 0. \\
\end{cases}$$
\end{definition}  


Then we have  
$$F_h : \cfrac{1}{A_1 + \cfrac{1}{A_2 + \ddots}} \mapsto 
\cfrac{1}{[gA_1]} + \cfrac{1}{A_2 + \ddots},
\qquad \cfrac{1}{a_0 + \cfrac{1}{A_2 + \ddots}} \mapsto \cfrac{1}{ A_2 + \ddots}, \cdots$$

\begin{example}\label{exm1}
Let $h = t^{-1}$. 
For an example, put $$f = \frac{1}{2t^3 + t^2 + 2 + r}, \quad \deg(r) < 0. $$ Then we have
\begin{equation*}
F_h(f) = \frac{1}{2t^2 + t + r}, \quad
F_h^2(f) = \frac{1}{2t + 1 + r}, \quad
F_h^3(f) = \frac{1}{2 + r}, \quad F_h^4(f) = r. 
\end{equation*}
\end{example}

Clearly we have $$ F_h^{-\deg (f) + 1} (f) = \Psi (f),$$
where $\Psi$ is the Artin map.

\begin{figure}
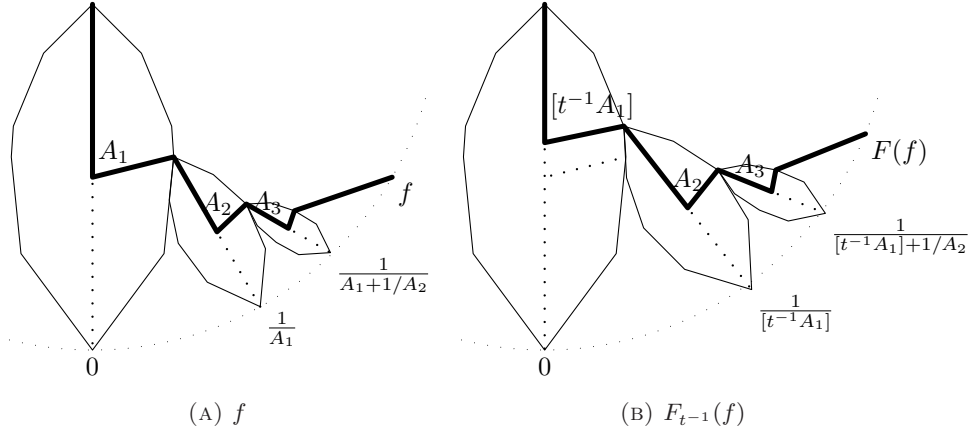

	\begin{subfigure}{0.48\textwidth}
		\centering
		\includegraphics{diagram-21.mps}
		\caption{$f$}
	\end{subfigure}
	\begin{subfigure}{0.48\textwidth}
		\centering
		\includegraphics{diagram-23.mps}
		\caption{$F_{t^{-1}}(f)$}
	\end{subfigure}
	\caption{The action of the Farey map $F_{t^{-1}}$ on the tree}
	\label{fig:2}
\end{figure}



Let us give a geometric motivation of the Farey map  $F_{h}$ defined
above. In Figure~\ref{fig:2},
the thick line represents the vertical geodesic from $\infty$ to $f$, which intersects Ford balls based on $P_0/Q_0 = 0/1$, $P_1/Q_1 = 1/A_1$, $P_2/Q_2=1/(A_1 + 1/A_2)$ and so on, by Lemma~\ref{lem:2}.
The geodesic to $F_{t^{-1}}(f)$ follows the path determined by $A_i, i=2, \cdots$ in each Ford ball except for the first Ford ball where it follows the path determined by $[t^{-1}A_1]$. More generally, the geodesic to $F_h(f)$ follows the path determined by $A_i, i=2, \cdots$ in each Ford ball except for the first Ford ball where it follows the path determined by $[h A_1]$.




\subsection{Farey map and intermediate convergents}\label{sec:4.2} 
In this subsection, we explain that since $K$ is algebraically closed, we obtain all the
intermediate convergents via Farey maps.


Let $M_h (f)$ be the matrix defined by 
$$M_h (f) = \begin{cases}
\begin{pmatrix} 1 & 0 \\ [(1-h) f^{-1}] & 1 \end{pmatrix}, &\text{ if } \deg (f) < 0, \\
\begin{pmatrix} 0 & 1 \\ 1 & [ (1-h) f^{-1}] \end{pmatrix}, &\text{ if } \deg (f) = 0. 
\end{cases}$$
Then the geodesic to $f$ corresponds to the sequence of matrices
$$ M_h(f) M_h(F_h(f)) M_h(F_h^2(f)) \cdots $$

\begin{proposition}
For each $f \in \mathbb L$  with $$\left[\frac 1f \right] = A_{1} = a_mt^m + a_{m-1}t^{m-1} + \dots + a_0, \ m = -\deg(f),$$ we have
$$ M_h(f) M_h(F_h(f)) \cdots M_h(F_h^{-\deg(f)}(f)) = \begin{pmatrix} 0 & 1 \\ 1 & A_1 \end{pmatrix} .$$
Moreover, if $\ell = \deg(A_1) + 1 + \deg(A_2) + 1 + \dots + \deg(A_k) + 1 + i$, $0 \le i \le \deg(A_{k+1})$, then 
\begin{equation*} \begin{split} M_h(f) \cdots M_h(F_h^{\ell-1}(f)) 
&= \begin{pmatrix} P_{k-1} & P_k \\ Q_{k-1} & Q_k \end{pmatrix} 
\begin{pmatrix} 1 & 0 \\ [ ( 1 - h^i ) A_{k+1} ] & 1 \end{pmatrix} \\
& 
= \begin{pmatrix} P_{k+1} - [ h^i A_{k+1} ] P_k  & P_k \\ Q_{k+1} - [ h^i A_{k+1} ] Q_k  & Q_k \end{pmatrix}.
\end{split} \end{equation*}
\end{proposition}

For $1 \le i \le \deg(A_{k+1})$, denote 
\begin{equation*}
U^h_{k,i} = P_{k+1} - [ h^i A_{k+1} ] P_k, \qquad
V^h_{k,i} = Q_{k+1} - [ h^i A_{k+1} ] Q_k.
\end{equation*}
We call $\frac{U^h_{k,i}}{V^h_{k,i}}$ the intermediate convergent of $f$ with respect to $h$.


From Theorem~\ref{thm:3.1}, it follows that for $1 \le i \le \deg(A_{k+1})$, we have 
$$
\left| f - \frac{U^h_{k,i}}{V^h_{k,i}} \right| = \frac{q^{-i}}{|Q_{k+1}| \cdot |Q_{k}|}
$$


Let us recall that the Farey map $F_\mathbb{J}$ of Nakada et al.\cite{BerNakNat} is defined as 
$$ 
F_{\mathbb J}(f) = \begin{cases} \dfrac{1}{G(f)}, &\text{ if } \deg G(f)  \ge 0, \\
\dfrac 1f - \left[ \dfrac 1f \right], &\text{ if } \deg G(f) <0. \end{cases} $$
Here,
$$ G(f) = \frac 1f - \frac{1}{LT(f)}$$
with $LT(f)$ being the leading term of $f$.

\begin{proposition}
For each $f$, there exist $s \in \mathbb N$ and $h \in \bL$ with $\deg(h) = -1$ such that $F^s_{h} (f) = F_{\mathbb{J}}(f)$.
\end{proposition}

\begin{proof}
If $\deg G(f) < 0$, then $F_{\mathbb J}(f) = \Psi(f) = F_h^{-\deg(f)+1} (f)$ for any $h \in \bL$ with $\deg(h) = -1$.

Assume that $\deg G(f) \ge 0$.
Let 
$g = 1- f \cdot LT(f^{-1}) = (1/f - LT(f^{-1})) \cdot f = G(F) \cdot f$
and $s = - \deg(g) \le - \deg(f)$.
Then there exists $h \in \K$ with $\deg(h) = -1$ such that $h^s = g$.
Let $A_1 = [1/f]$. Then we have
$$[h^s A_1] = [g A_1]= [( 1- f \cdot LT(f^{-1}) )A_1] = A_1 - [A_1 f LT( A_1)] = A_1 - [LT(A_1)].$$ 
\end{proof}


By Theorem~\ref{thm:3.1}, we immediately have the following:
\begin{theorem}
If $P/Q \in \F(t)$ with $\deg(Q) = \deg(Q_{k+1})$ satisfies
$$
\left| f - \frac{P}{Q} \right| < \frac{1}{|Q_{k+1}|\cdot |Q_k|},
$$
then we have 
$$\frac PQ = \frac{U^{h}_{k,i}}{V^{h}_{k,i}},$$ for some $h$ and $k,i$.
\end{theorem}

\subsection{Ergodic Theory of the Farey map}
Let $\mu$ be the Haar measure of $\mathbb F_{q}((t^{-1}))$ 
normalized as $\mu (\mathcal O) = 1$.
For each $n \ge 0$, denote  
$$\mathbb J_n = \{  f \in \mathcal O : \deg(f) = - n \}.$$ 
Define a measure $\mu_A$ on $\mathcal O$ given by 
$$
\mu_A ({\bf D}) = \frac{q^2}{2q-1} \cdot \mu ({\bf D} \cap \mathbb L) + \frac{q}{2q-1} \cdot \mu ( {\bf D} \cap \mathbb J_0),
$$
for each Borel set ${\bf D} \subset \mathcal O$. 
Then for each $h$, the probability measure $\mu_A$ on $\mathcal O$ is an ergodic invariant measure for the Farey map $F_h$.

\begin{proof}[Proof of Theorem~\ref{thm2}]
Suppose that $\bf D$ is a Borel subset of $\mathbb L$ and 
$P \in \mathbb F_{q}[t]$ 
with $\deg (P) = k \ge 0$.  We consider
\[
P+{\bf D} = \{ P + r \in \K :  r \in {\bf D}\},
\quad 
(P+{\bf D})^{-1} = \{f \in \K :  f^{-1} \in P + {\bf D} \}.
\]
Then we see 
\[
\mu( (P + {\bf D})^{-1} ) = \frac{1}{q^{2k}} \mu ( P + {\bf D}) . 
\]
For any Borel set $\bf D$ of $\mathcal O$, we can decompose it as a disjoint 
union such that 
\[
  {\bf D} = \bigcup_{ k=0}^{\infty} \bigcup_{P \in F_{q}[t], \deg P = k} (P + {\bf B}_P)^{-1} . 
\]
In this sense, it is enough to show $\mu_A (F_{h}^{-1}{\bf D}) 
= \mu_A({\bf D})$ 
for ${\bf D}$ of the form $(P + {\bf B})^{-1}$ with $P \in \mathbb F_{q}[t]$, 
$\deg P = k \ge0$. 
First we assume that $k = 0$. Then ${\bf D}$ is of the form 
$(a + {\bf B})^{-1}$ with 
a Borel set ${\bf B} \subset \mathbb L$ and $a \in \mathbb F_{q}^{\ast}$. 
For $f \in F_{h}^{-1} {\mathbf D}$, $F_{h}(f)$ is 
$\frac{1}{ 1/f - (b_{1} t + b_{0} - b_{1} h_{1})}$ 
where $[ 1/f ] = b_{1} t + b_{0}$ and $h_{1}$ is the leading 
coefficient of $h$.  This implies $b_{1} h_{1} = a$ and thus $b_{1}$ is 
uniquely determind when $a$ is fixed, on the otherhand, $b_{0}$ is free. 
This shows 
\[
F_{h}^{-1} {\mathbf D} = \bigcup_{b_{0} \in \mathbb F_{q}} \frac{1}{ a (h_{1})^{-1} t + b_{0} + {\bf B}}
\] 
and thus 
\begin{align*}
\frac{q^{2}}{2q-1} \mu (F_{h}^{-1}{\bf D}) &= \frac{q^{2}}{2q-1} \sum_{b_{0} \in \mathbb F_{q}} \frac{1}{q^{2}} \mu ({\bf B})
= \frac{q}{2q-1}  \mu (\mathbf B) \\
&= \frac{q}{2q-1}  \mu(a + {\bf B}) = \frac{q}{2q-1} \mu({\bf D}) , 
\end{align*}
which means $\mu_A (F_{h}^{-1}{\bf D}) = \mu_A ({\bf D})$. 

Next we assume that $k > 0$. By the similar way, we see that 
$f \in F_{h}^{-1}{\bf D} \cap \mathbb J_{k+1}$ $F_{h}(f)$ is of the form 
$ \frac{1}{P'+ {\bf B}}$ and the coefficients of $P'$ are completely fixed by $h$ and $P$ except for the constant term. Thus we have 
\[
\mu_A (F_{h}^{-1}{\bf D} \cap \mathbb J_{k+1}) = \frac{1}{q}\mu_A ({\bf D}). 
\]  
On the other hand, $f \in F_{h}^{-1}{\bf D} \cap \mathbb J_{0}$ is equivalent 
to $f \in \cup_{a \in \F^{\ast}} \left( a + \frac{1}{P + {\bf D}} \right)^{-1}$. 
Here 
\[
\frac{q}{2q -1} \mu \left( \frac{1}{a+ \frac{1}{P + {\bf D}}} \right)
= \frac{q}{2q -1}\mu \left( a+ \frac{1}{P + {\bf D}} \right) 
= \frac{q}{2q -1}\mu \left(\frac{1}{P + {\bf D}} \right).
\]
Thus 
\[
\mu_A (F_{h}^{-1}{\bf D} \cap \mathbb J_{0}) = \frac{q-1}{q}\mu_A({\bf D}).
\]
Consequently, we have 
\[
\mu_A (F_{h}^{-1}{\bf D} ) = \mu_A({\bf D}) . 
\]

Similarly with the proof of Theorem~\ref{thm1}, the ergodicity of $F_{h}$ with respect to $\mu_A$ is an easy consequence of the fact that the Artin map is a jump transformation of 
$F_{h}$ and that the Artin map is ergodic with respect to the Haar measure, see \cite{Schw}.
\end{proof}

Suppose that $A_{k} \in \mathbb F_{q}[t]$ is the $k$-th coefficient continued 
expansion of $f \in \mathbb L$.  
Let's write $\begin{pmatrix} U_{\ell} \\ V_{\ell}\end{pmatrix}$ 
the first column of $M_{h}(f) \cdots M_{h}(F_{h}^{\ell}(f))$.  If 
$\ell = \sum_{n=1}^{k} \deg A_{n} \, + \, k$, then it is 
$\begin{pmatrix} P_{k-1} \\ Q_{k-1}\end{pmatrix}$, i.e. $k$-th convergent of the continued fraction expansion of $f$.  Otherwise, 
$\begin{pmatrix} U_{k, i}^{h} \\ V_{k, i}^{h}\end{pmatrix}$ for 
$\ell =  \sum_{n=1}^{k} \deg A_{n} \, + \, k \, + i$ 
with $1 \le i \le \deg A_{k+1} $.
Then for $\mu$-almost every $f$, we have 
\[
\lim_{\ell \to \infty}
\frac{1}{\ell} \log_{q} \left| f \, - \, \frac{U_{\ell}}{V_{\ell}} \right| 
\, = \, -\frac{2q}{2q - 1}
\]

\begin{proof}[Proof of Theorem~\ref{thm3}]
We see 
\[
\frac{1}{\ell} \log_{q} \left| f \, - \, \frac{U_{\ell}}{V_{\ell}} \right| 
\, = \, - \frac{2 \sum_{n=1}^{k} \deg A_{n} \, + \, \deg A_{k+1} \, + \, i}
      {\ell}
\]
for $\ell =  \sum_{n=1}^{k} \deg A_{n} \, + \, k \, + i$ 
with $1 \le i \le \deg A_{k+1} $.  In this case, the right side is 
\[
- \frac{2 \sum_{n=1}^{k} \deg A_{n} \, + \, \deg A_{k+1} \, + \, i}
      {\sum_{n=1}^{k} \deg A_{n} \, + \, k \, + i}.
\]
For $\mu$-almost every $f$, we have (see \cite{BeNa})
\[
\lim_{k \to \infty} \frac{\sum_{n=1}^{k} \deg A_{n}}{k} = \frac{q}{q-1}
\]
and 
\[
\lim_{k \to \infty} \frac{\deg A_{k+1}}{k} = 0 .
\]
Thus we have 
\[
\frac{1}{\ell} \log_{q} \left| f \, - \, \frac{U_{\ell}}{V_{\ell}} \right| 
\]
converges to 
\[
- \frac{2 q}{2q - 1} . 
\]
along $\ell =  \sum_{n=1}^{k} \deg A_{n} \, + \, k \, + i$ 
with $1 \le i \le \deg A_{k+1} $.  If 
$\ell =  \sum_{n=1}^{k} \deg A_{n} \, + \, k$, then it is easy to see that the 
same holds.  Altogether we have 
\[
\lim_{\ell \to \infty}
\frac{1}{\ell} \log_{q} \left| f \, - \, \frac{U_{\ell}}{V_{\ell}} \right| 
\, = \, - \frac{2 q}{2q - 1} .
\]
\end{proof}

\section*{Acknowledgement}
Dong Han Kim is supported by KRF 2012R1A1A2004473. Seonhee Lim is supported by KRF 2012-000-8829, KRF 2012-000-2388.
Hitoshi Nakada is supported in part by the Grant-in-Aid for Scientific research (No. 24340020), JSPS.
Rie Natsui is supported in part by the Grant-in-Aid for Scientific research (No. 23740088), JSPS.

\end{document}